\documentclass[a4paper, 10pt]{article}
\usepackage[cp1251]{inputenc}
\usepackage[english]{babel}
\usepackage{amsthm}
\usepackage{cite}
\usepackage{tikz-cd}
\usepackage{amsmath}
\usepackage{amsfonts}
\usepackage{amssymb}
\usepackage{hyperref}
\usepackage[]{authblk}
\usepackage{epsfig,graphicx}

\DeclareMathOperator{\Aut}{Aut}
\DeclareMathOperator{\End}{End}
\DeclareMathOperator{\Char}{char}
\DeclareMathOperator{\Det}{det}

\DeclareMathOperator{\Tr}{tr}

\DeclareMathOperator{\Deg}{deg}

\DeclareMathOperator{\Sp}{Sp}
\DeclareMathOperator{\Ht}{ht}
\DeclareMathOperator{\J}{J}

\DeclareMathOperator{\GL}{GL}
\newcommand{\TAut}{\operatorname{TAut}}

\newtheorem{thm}{Theorem}[section]
\newtheorem{lem}[thm]{Lemma}

\newtheorem{prop}[thm]{Proposition}

\newtheorem{conj}[thm]{Conjecture}

\begin{document}
\renewcommand{\thefootnote}{\fnsymbol{footnote}}
\renewcommand{\thefootnote}{\arabic{footnote}}
\fontsize{11}{11pt}\selectfont
\title{\bf Lifting of Polynomial Symplectomorphisms and Deformation Quantization}
\renewcommand\Affilfont{\itshape\small}
\author[1]{Alexei Kanel-Belov\thanks{kanel@mccme.ru}}
\author[2]{Sergey Grigoriev\thanks{gregomaths@rambler.ru}}
\author[3]{Andrey Elishev\thanks{elishev@phystech.edu}}
\author[4]{Jie-Tai Yu\thanks{jietaiyu@szu.edu.cn}}
\author[5]{Wenchao Zhang\thanks{whzecomjm@gmail.com}}
\affil[1]{College of Mathematics and Statistics, Shenzhen University, Shenzhen, 518061, China}
\affil[2]{Department of Mechanics and Mathematics, Lomonosov Moscow State University, Vorobievy Gory, Moscow, 119898, Russia}
\affil[3]{Laboratory of Advanced Combinatorics and Network Applications, Moscow Institute of Physics and Technology, Dolgoprudny, Moscow Region, 141700, Russia}
\affil[4]{Shenzhen University, Shenzhen, 518061, China}
\affil[5]{Mathematics Department, Bar-Ilan University, Ramat-Gan, 52900, Israel}

\date{}

\maketitle

\renewcommand{\abstractname}{Abstract}
\begin{abstract}

We study the problem of lifting of polynomial symplectomorphisms in characteristic zero to automorphisms of the Weyl algebra by means of approximation by tame automorphisms. In 1983, Anick proved the fundamental result on approximation of polynomial automorphisms. We obtain similar approximation theorems for symplectomorphisms and Weyl algebra authomorphisms. We then formulate the lifting problem. More precisely, we prove the possibility of lifting of a symplectomorphism to an automorphism of the power series completion of the Weyl algebra of the corresponding rank. The lifting problem has its origins in the context of deformation quantization of the affine space and is closely related to several major open problems in algebraic geometry and ring theory.

This paper is a continuation of the study \cite{KBRZh}.

\end{abstract}

\section{Introduction}

One of the active research areas in ring theory concerns the geometry of polynomial endomorphisms -- that is, endomorphisms of finitely generated associative algebras (typically over a field $\mathbb{K}$) subject to a set of polynomial identities and possibly carrying other structures. Arguably, the most renowned -- and notoriously difficult -- open problems in this area is the Jacobian conjecture of Keller \cite{vdE}, open for all $N\geq 2$:

\begin{conj}
If $\mathbb{K}$ is a field of characteristic zero and $\varphi:\mathbb{A}^N_{\mathbb{K}}\rightarrow \mathbb{A}^N_{\mathbb{K}}$ is a polynomial mapping of the affine space of dimension $N$ with unit Jacobian:
\begin{equation*}
\J(\varphi)= \Det \left[\frac{\partial \varphi(x_i)}{\partial x_j}\right]=1
\end{equation*}
 then $\varphi$ is invertible (and the inverse is also a polynomial mapping).
\end{conj}

Tsuchimoto \cite{Tsu1, Tsu2}, and independently Kanel-Belov and Kontsevich \cite{BKK2}, found a deep connection between the Jacobian conjecture and a celebrated conjecture of Dixmier \cite{Dix} on endomorphisms of the Weyl algebra, which is stated as follows:

\begin{conj}
Any endomorphism $\phi$ of the $n$-th Weyl algebra $W_n(\mathbb{K})$ in characteristic zero is invertible.
\end{conj}

\medskip

The correspondence between the two open problems, in the case of algebraically closed $\mathbb{K}$, is based on the existence of a composition-preserving map
\begin{equation*}
\End W_n(\mathbb{K})\rightarrow \End \mathbb{K}[x_1,\ldots,x_{2n}]
\end{equation*}
which is a homomorphism for the corresponding automorphism groups. Furthermore, the mappings that belong to the image of this homomorphism preserve the standard symplectic form on the $2n$-dimensional affine space $\mathbb{A}^{2n}_{\mathbb{K}}$.
In accordance with this, Kontsevich and Kanel-Belov \cite{BKK1} formulated several conjectures on correspondence between automorphisms of the Weyl algebra $W_n$ and the Poisson algebra $P_n$ (which is the polynomial algebra $\mathbb{K}[x_1,\ldots,x_{2n}]$ endowed with the standard Poisson bracket) in characteristic zero. In particular, there is a
\begin{conj}
The automorphism groups of the $n$-th Weyl algebra and the polynomial algebra in $2n$ variables with Poisson structure over the rational numbers are isomorphic:
\begin{equation*}
\Aut W_n(\mathbb{Q})\simeq \Aut P_n(\mathbb{Q})
\end{equation*}
\end{conj}
Relatively little is known about the case $\mathbb{K}=\mathbb{Q}$, and the proof techniques developed in \cite{BKK1} rely heavily on model-theoretic objects such as infinite prime numbers (in the sense of non-standard analysis); that in turn requires the base field $\mathbb{K}$ to be of characteristic zero and algebraically closed (effectively $\mathbb{C}$ by the Lefschetz principle). However, even the seemingly easier analogue of the above conjecture, the case $\mathbb{K}=\mathbb{C}$, is known (and positive) only for $n=1$.

\medskip

In the case $n=1$, the affirmative answer to the Kontsevich conjecture, as well as positivity of several isomorphism statements for algebras of similar nature, relies on the fact that all automorphisms of the algebras in question are tame (see definition below). Groups of tame automorphisms are rather interesting objects. Anick \cite{An} has proved that the group of tame automorphisms of $\mathbb{K}[x_1,\ldots,x_N]$ is dense (in power series topology) in the subspace of all endomorphisms with non-zero constant Jacobian. This fundamental result enables one to reformulate the Jacobian conjecture as a statement on invertibility of limits of tame automorphism sequences.

\medskip

Another interesting problem is to ask whether all automorphisms of a given algebra are tame \cite{Jung, VdK, Czer, Shes3, Shes4}. For instance, it is the case \cite{ML2, ML3} for $\mathbb{K}[x,y]$, the free associative algebra $\mathbb{K}\langle x,y\rangle$ and the free Poisson algebra $\mathbb{K}\lbrace x,y\rbrace$. It is also the case for free Lie algebras (a result of P. M. Cohn). On the other hand, tameness is no longer the case for $\mathbb{K}[x,y,z]$ (the wild automorphism example is provided by the well-known Nagata automorphism, cf. \cite{Nagata1, Shes2}).

\medskip

Anick's approximation theorem was established for polynomial automorphisms in 1983. We obtain the approximation theorems for polynomial symplectomorphisms and Weyl algebra automorphisms. These new cases are established after more than 30 years. The focus of this paper is \textbf{the problem of lifting of symplectomorphisms}:

\medskip

-- \emph{can an arbitrary symplectomorphism in dimension $2n$ be lifted to an automorphism of the $n$-th Weyl algebra in characteristic zero?}

\medskip

The lifting problem is the milestone in the Kontsevich conjecture. The use of tame approximation is advantageous due to the fact that tame symplectomorphisms correspond to Weyl algebra automorphisms: in fact \cite{BKK1}, the tame automorphism subgroups are isomorphic when $\mathbb{K}=\mathbb{C}$.

\medskip

The problems formulated above, as well as other statements of similar flavor, outline behavior of algebro-geometric objects when subject to quantization. Conversely, quantization (and anti-quantization in the sense of Tsuchimoto) provides a new perspective for the study of various properties of classical objects; many of such properties are of distinctly K-theoretic nature. The lifting problem is a subject of a thorough study of Artamonov \cite{Art78, Art84, Art91, Art98}, one of the main results of which is the proof of an analogue of the Serre -- Quillen -- Suslin theorem for metabelian algebras. The possibility of lifting of (commutative) polynomial automorphisms to automorphisms of metabelian algebra is a well-known result of Umirbaev, cf. \cite{Umir}; the metabelian lifting property was instrumental in Umirbaev's resolution of the Anick's conjecture (which says that a specific automorphism of the free algebra $\mathbb{K}\langle x,y,z\rangle$, $\Char \mathbb{K}=0$ is wild). Related to that also is a series of well-known papers \cite{Shes1, Shes2, Shes4, Shes3}.

\medskip

An interesting and in a sense essential generalization of this line of inquiry is obtained by taking it to the realm of quantum algebra. Indeed, the algebraic and K-theoretic language of quantization can be extended naturally to account for the relevant non-commutative geometry. Accordingly, the vast majority of problems formulated above may also be posed in the quantum algebraic context. Automorphism groups of algebras of quantum polynomials were the subject of an investigation of Artamonov \cite{Art02, Art05}. Such algebras provide a generalization of the Weyl algebra, and it is a question of legitimate interest whether the topology of the corresponding automorphism group allows for approximation theorems analogous to the ones discussed in the present paper.


\medskip

We establish the approximation property for polynomial symplectomorphisms and comment on the lifting problem of polynomial symplectomorphisms and Weyl algebra automorphisms. In particular, the main results discussed here are as follows.


\medskip

\textbf{Main Theorem 1.} \emph{Let $\varphi=(\varphi(x_1),\;\ldots,\;\varphi(x_N))$ be an automorphism of the polynomial algebra $\mathbb{K}[x_1,\ldots,x_N]$ over a field $\mathbb{K}$ of characteristic zero, such that its Jacobian}
\emph{is equal to $1$. Then there exists a sequence $\lbrace \psi_k\rbrace\subset \TAut \mathbb{K}[x_1,\ldots,x_N]$ of tame automorphisms which converges to $\varphi$ in formal power series topology.}

\medskip

Anick \cite{An} proved the tame approximation theorem for polynomial automorphisms. We present a slightly modified elementary proof of Anick's theorem, which is then adapted to the problem of approximation by polynomial symplectomorphisms. That in turn allows us to attack the problem of approximating Weyl algebra automorphisms -- as tame symplectomorphisms have (unique) preimages under the Kanel-Belov -- Kontsevich homomorphism \cite{BKK1}. 

\medskip

\textbf{Main Theorem 2.}\emph{ Let $\sigma=(\sigma(x_1),\;\ldots,\;\sigma(x_n),\;\sigma(p_1),\;\ldots,\;\sigma(p_n))$ be a symplectomorphism of $\mathbb{K}[x_1,\ldots,x_n,p_1,\ldots,p_n]$ with unit Jacobian.}
\emph{Then there exists a sequence $\lbrace \tau_k\rbrace\subset \TAut P_n(\mathbb{K})$ of tame symplectomorphisms which converges to $\sigma$ in formal power series topology.}

\medskip

\textbf{Main Theorem 3.} \emph{Let $\mathbb{K}=\mathbb{C}$ and let $\sigma:P_n(\mathbb{C})\rightarrow P_n(\mathbb{C})$ be a symplectomorphism over complex numbers. Then there exists a sequence}
\begin{equation*}
\psi_1,\;\psi_2,\;\ldots,\;\psi_k,\;\ldots
\end{equation*}
\emph{of tame automorphisms of the $n$-th Weyl algebra $W_n(\mathbb{C})$, such that their images $\sigma_k$ in $\Aut P_n(\mathbb{C})$ converge to $\sigma$.}
\medskip

The last theorem is of main concern to us. As we shall see, sequences of tame symplectomorphisms lifted to automorphisms of Weyl algebra (either by means of the isomorphism of \cite{BKK1}, or explicitly through deformation quantization $P_n(\mathbb{C})\rightarrow P_n(\mathbb{C})[[\hbar]]$) are such that their limits may be thought of as power series in Weyl algebra generators. If we could establish that those power series were actually polynomials, then the Dixmier conjecture would imply the Kontsevich's conjecture (with $\mathbb{Q}$ replaced by $\mathbb{C}$). Conversely, approximation by tame automorphisms provides a possible means to attack the Dixmier conjecture (and, correspondingly, the Jacobian conjecture).

\medskip

Another important detail of approximation by tame automorphisms is its natural behavior with respect to the $\mathfrak{m}$-adic topology on local rings of automorphism varieties. This is formalized in the following two results.

\medskip

\textbf{Main Theorem 4.} Let $\varphi$ be a polynomial automorphism and let $\mathcal{O}_{\varphi}$ be the local ring of $\Aut \mathbb{C}[x_1,\ldots,x_n]$ with its maximal ideal $\mathfrak{m}$. Then there exists a tame sequence $\lbrace \psi_k\rbrace$ which converges to $\varphi$ in power series topology, such that the coordinates of $\psi_k$ converge to coordinates of $\varphi$ in $\mathfrak{m}$-adic topology.

\medskip

\textbf{Main Theorem 5.} Let $\sigma$ be a symplectomorphism and let $\mathcal{O}_{\sigma}$ be the local ring of $\Aut P_n(\mathbb{C})$ with its maximal ideal $\mathfrak{m}$. Then there exists a sequence of tame symplectomorphisms $\lbrace \sigma_k\rbrace$ which converges to $\sigma$ in power series topology, such that the coordinates of $\sigma_k$ converge to coordinates of $\sigma$ in $\mathfrak{m}$-adic topology.

\medskip

The present paper serves as a continuation and expansion of our previous study of quantization \cite{KBRZh}.

\section{Endomorphisms of $\mathbb{K}[x_1,\ldots,x_n]$, $W_n(\mathbb{K})$ and $P_n(\mathbb{K})$}
\subsection{Definitions and notation}
The $n$-th Weyl algebra $W_{n}(\mathbb{K})$ over $\mathbb{K}$ is by definition the quotient of the free associative algebra
\begin{equation*}
\mathbb{K}\langle a_1,\ldots,a_n,b_1,\ldots,b_n\rangle
\end{equation*}
by the two-sided ideal generated by elements
\begin{equation*}
b_ia_j-a_jb_i-\delta_{ij},\;\;a_ia_j-a_ja_i,\;\;b_ib_j-b_jb_i,
\end{equation*}
with $1\leq i,j\leq n$. One can think of $W_{n}(\mathbb{K})$ as the algebra
\begin{equation*}
\mathbb{K}[x_1,\ldots,x_n,y_1,\ldots,y_n]
\end{equation*}
with two sets of $n$ mutually commuting generators (images of the free generators under the canonical projection) which interact according to $[y_i,x_j]=y_ix_j-x_jy_i=\delta_{ij}$; henceforth we denote the Weyl algebra by $W_{n}(\mathbb{K})$ in order to avoid confusion with $\mathbb{K}[X]$ -- notation reserved for the ring of polynomials in commuting variables.

\medskip

The polynomial algebra $\mathbb{K}[x_1,\ldots,x_N]$ itself is the quotient of the free associative algebra by the congruence that makes all its generators commutative. When the number $N$ of generators is even, the  algebra $\mathbb{K}[x_1,\ldots,x_{2n}]$ carries an additional structure of the Poisson algebra -- namely, a bilinear map $$\lbrace \;,\;\rbrace:\mathbb{K}[x_1,\ldots,x_{2n}]\otimes \mathbb{K}[x_1,\ldots,x_{2n}]\rightarrow \mathbb{K}[x_1,\ldots,x_{2n}]$$ that turns $\mathbb{K}[x_1,\ldots,x_{2n}]$ into a Lie algebra and acts as a derivation with respect to polynomial multiplication. Under a fixed choice of generators, this map is given by the standard Poisson bracket
\begin{equation*}
\lbrace x_i,x_j\rbrace = \delta_{i,n+j}-\delta_{i+n,j}.
\end{equation*}
We denote the pair $(\mathbb{K}[x_1,\ldots,x_{2n}],\lbrace \;,\;\rbrace)$ by $P_n(\mathbb{K})$. In our discussion the coefficient ring $\mathbb{K}$ is a field of characteristic zero, and for later purposes (Proposition 4.3) we require $\mathbb{K}$ to be algebraically closed. Thus one may safely assume $\mathbb{K}=\mathbb{C}$ in the sequel.

\medskip

Throughout we assume all homomorphisms to be unital and preserving all defining structures carried by the objects in question. Thus, by a Weyl algebra endomorphism we always mean a $\mathbb{K}$-linear ring homomorphism $W_{n}(\mathbb{K})$ into itself that maps $1$ to $1$. Similarly, the set $\End \mathbb{K}[x_1,\ldots,x_n]$ consists of all $\mathbb{K}$-endomorphisms of the polynomial algebra, while $\End P_{n}$ is the set of polynomial endomorphisms preserving the Poisson structure. We will call elements of the group $\Aut P_{n}$ \textbf{polynomial symplectomorphisms}, due to the fact that they can be identified with polynomial one-to-one mappings $\mathbb{A}^{2n}_{\mathbb{K}}\rightarrow\mathbb{A}^{2n}_{\mathbb{K}}$ of the affine space $\mathbb{A}^{2n}_{\mathbb{K}}$ which preserve the symplectic form
\begin{equation*}
\omega=\sum_{i}dp_i\wedge dx_i.
\end{equation*}

\medskip

Any endomorphism $\varphi$ of $\mathbb{K}[x_1,\ldots,x_N]$, $P_n(\mathbb{K})$ or $W_n(\mathbb{K})$ can be identified with the ordered set
\begin{equation*}
(\varphi(x_1),\;\varphi(x_2),\;\ldots)
\end{equation*}
of images of generators of the corresponding algebra. For $\mathbb{K}[x_1,\ldots,x_N]$ and $P_n(\mathbb{K})$, the polynomials $\varphi(x_i)$ can be decomposed into sums of homogeneous components; this means that the endomorphism $\varphi$ may be written as a formal sum
\begin{equation*}
\varphi = \varphi_0+\varphi_1+\cdots,
\end{equation*}
where $\varphi_k$ is a string (of length $N$ and $2n$, respectively) whose entries are homogeneous polynomials of total degree $k$.\footnote{We set $\Deg x_i=1$.} Accordingly, the height $\Ht(\varphi)$ of the endomorphism is defined as
\begin{equation*}
\Ht(\varphi)=\inf\lbrace k\;|\;\varphi_k\neq 0\rbrace,\;\;\Ht(0)=\infty.
\end{equation*}
This is not to be confused with the degree of endomorphism, which is defined as $\Deg(\varphi)=\sup\lbrace k\;|\;\varphi_k\neq 0\rbrace$.\footnote{For $W_n$ the degree is well defined, but the height depends on the ordering of the generators.} The height $\Ht(f)$ of a polynomial $f$ is defined quite similarly to be the minimal number $k$ such that the homogeneous component $f_k$ is not zero. Evidently, for an endomorphism $\varphi=(\varphi(x_1),\;\ldots,\;\varphi(x_N))$ one has
\begin{equation*}
\Ht(\varphi)=\inf\lbrace \Ht(\varphi(x_i))\;|\;1\leq i\leq N\rbrace.
\end{equation*}

\medskip

The function
\begin{equation*}
d(\varphi,\psi)=\exp(-\Ht(\varphi-\psi))
\end{equation*}
is a metric on $\End \mathbb{K}[x_1,\ldots,x_N]$. We will refer to the corresponding topology on $\End$ (and on subspaces such as $\Aut$ and $\TAut$) as the formal power series topology.

\medskip

\subsection{Tame automorphisms}
We call an automorphism $\varphi\in\Aut\mathbb{K}[x_1,\ldots,x_N]$ \textbf{elementary} if it is of the form
\begin{equation*}
\varphi = (x_1,\ldots,\;x_{k-1},\;ax_k+f(x_1,\ldots,x_{k-1},\;x_{k+1},\;\ldots,\;x_N),\;x_{k+1},\;\ldots,\;x_N)
\end{equation*}
with $a\in\mathbb{K}^{\times}$. Observe that linear invertible changes of variables -- that is, transformations of the form
\begin{equation*}
(x_1,\;\ldots,\;x_N)\mapsto (x_1,\;\ldots,\;x_N)A,\;\;A\in\GL(N,\mathbb{K})
\end{equation*}
are realized as compositions of elementary automorphisms.

The subgroup of $\Aut\mathbb{K}[x_1,\ldots,x_N]$ generated by all elementary automorphisms is the group $\TAut \mathbb{K}[x_1,\ldots,x_N]$ of so-called \textbf{tame automorphisms}.

\medskip

Let $P_{n}(\mathbb{K})=\mathbb{K}[x_1,\ldots,x_{n},p_1,\ldots,p_n]$ be the polynomial algebra in $2n$ variables with Poisson structure. It is clear that for an elementary $\varphi\in\Aut\mathbb{K}[x_1,\ldots,x_{n},p_1,\ldots,p_n]$ to be a symplectomorphism, it must be either a linear symplectic change of variables -- that is, a transformation of the form
\begin{equation*}
(x_1,\;\ldots,\;x_n,\;p_1,\;\ldots,\;p_n)\mapsto (x_1,\;\ldots,\;x_n,\;p_1,\;\ldots,\;p_n)A
\end{equation*}
with $A\in\Sp(2n,\mathbb{K})$ a symplectic matrix, or an elementary transformation of one of two following types:
\begin{equation*}
(x_1,\;\ldots,\;x_{k-1},\;x_k+f(p_1,\;\ldots,\;p_n),\;x_{k+1},\;\ldots,\;x_n,\;p_1,\;\ldots,\;p_n)
\end{equation*}
and
\begin{equation*}
(x_1,\;\ldots,\;x_{n},\;p_1,\;\ldots,\;p_{k-1},\;p_k+g(x_1,\;\ldots,\;x_n),\;p_{k+1},\;\ldots,\;p_n).
\end{equation*}
Note that in both cases we do not include translations of the affine space into our consideration, so we may safely assume the polynomials $f$ and $g$ to be at least of height one.

The subgroup of $\Aut P_{n}(\mathbb{K})$ generated by all such automorphisms is the group $\TAut P_{n}(\mathbb{K})$ of \textbf{tame symplectomorphisms}. One similarly defines the notion of tameness for the Weyl algebra $W_n(\mathbb{K})$, with tame elementary automorphisms having the exact same form as for $P_n(\mathbb{K})$.

\medskip

The automorphisms which are not tame are called \textbf{wild}. It is unknown at the time of writing whether the algebras $W_n$ and $P_n$ have any wild automorphisms in characteristic zero for $n>1$, however for $n=1$ all automorphisms are known to be tame \cite{Jung, VdK, ML1, ML2}. On the other hand, the celebrated example of Nagata
\begin{equation*}
(x+(x^2-yz)x,\;y+2(x^2-yz)x+(x^2-yz)^2z,\;z)
\end{equation*}
provides a wild automorphism of the polynomial algebra $\mathbb{K}[x,y,z]$.

\medskip

It is known due to Kanel-Belov and Kontsevich \cite{BKK1,BKK2} that for $\mathbb{K}=\mathbb{C}$ the groups
\begin{equation*}
\TAut W_n(\mathbb{C})\;\;\text{and}\;\;\TAut P_n(\mathbb{C})
\end{equation*}
are isomorphic. The homomorphism between the tame subgroups is obtained by means of non-standard analysis and involves certain non-constructible entities, such as free ultrafilters and infinite prime numbers. Recent effort \cite{K-BE1, K-BE2} has been directed to proving the homomorphism's independence of such auxiliary objects, with limited success.

\section{Approximation by tame automorphisms}

Let $\varphi\in\Aut \mathbb{K}[x_1,\ldots,x_N]$ be a polynomial automorphism. We say that $\varphi$ is approximated by tame automorphisms if there is a sequence
\begin{equation*}
\psi_1,\;\psi_2,\ldots,\;\psi_k,\ldots
\end{equation*}
of tame automorphisms such that
\begin{equation*}
\Ht((\psi_k^{-1}\circ\varphi)(x_i)-x_i)\geq k
\end{equation*}
for $1\leq i\leq N$ and all $k$ sufficiently large. Observe that any tame automorphism $\psi$ is approximated by itself -- that is, by a stationary sequence $\psi_k=\psi$.

\medskip

This and the next section are dedicated to the proof of the first two main results stated in the introduction, which we reproduce here.

\begin{thm}
Let $\varphi=(\varphi(x_1),\;\ldots,\;\varphi(x_N))$ be an automorphism of the polynomial algebra $\mathbb{K}[x_1,\ldots,x_N]$ over a field $\mathbb{K}$ of characteristic zero, such that its Jacobian
\begin{equation*}
\J(\varphi)=\Det \left[\frac{\partial \varphi(x_i)}{\partial x_j}\right]
\end{equation*}
is equal to $1$. Then there exists a sequence $\lbrace \psi_k\rbrace\subset \TAut \mathbb{K}[x_1,\ldots,x_N]$ of tame automorphisms approximating $\varphi$.
\end{thm}

\begin{thm}
Let $\sigma=(\sigma(x_1),\;\ldots,\;\sigma(x_n),\;\sigma(p_1),\;\ldots,\;\sigma(p_n))$ be a symplectomorphism of $\mathbb{K}[x_1,\ldots,x_n,p_1,\ldots,p_n]$ with unit Jacobian.
Then there exists a sequence $\lbrace \tau_k\rbrace\subset \TAut P_n(\mathbb{K})$ of tame symplectomorphisms approximating $\sigma$.
\end{thm}

Theorem 3.1 is a special case of a classical result of Anick \cite{An} (Anick proved approximation for all \'etale maps, not just automorphisms). We give here a slightly simplified proof suitable for our context. The second theorem first appeared in \cite{SG} and is essential in our approach to the lifting problem in deformation quantization.

\medskip

The proof of Theorem 3.1 consists of several steps each of which amounts to composing a given automorphism $\varphi$ with a tame transformation of a specific type -- an operation which allows one to dispose in $\varphi(x_i)$ ($1\leq i\leq N$) of monomial terms of a given total degree, assuming that the lower degree terms have already been dealt with. Thus the approximating sequence of tame automorphisms is constructed by induction. As it was mentioned before, we disregard translation automorphisms completely: all automorphisms discussed here are origin-preserving, so that the polynomials $\varphi(x_i)$ have zero free part. This of course leads to no loss of generality.

\medskip

The process starts with the following straightforward observation.

\smallskip

\begin{lem}
There is a linear transformation $A\in\GL(N,\mathbb{K})$
\begin{equation*}
(x_1,\;\ldots,\;x_N)\mapsto (x_1,\;\ldots,\;x_N)A
\end{equation*}
such that its composition $\varphi_A$ with $\varphi$ fulfills
\begin{equation*}
\Ht(\varphi_A(x_i)-x_i)\geq 2
\end{equation*}
for all $i\in\lbrace 1,\ldots,N\rbrace$.
\end{lem}

\begin{proof}
Consider $$A_1=\left[\frac{\partial \varphi(x_i)}{\partial x_j}\right]\left( 0,\ldots,0\right)$$
-- the linear part of $\varphi$. Its determinant is equal to the value of $\J(\varphi)$ at zero, and $\J(\varphi)$ is a non-zero constant. Composing $\varphi$ with the linear change of variables induced by $A_1^{-1}$ (on the left) results in an automorphism $\varphi_A$ that is identity modulo $\mathit{O}(x^2)$.
\end{proof}

Using the above lemma, we may replace $\varphi$ with $\varphi_A$ (and suppress the $A$ subscript for convenience), thus considering automorphisms which are close to the identity in the formal power series topology.

\smallskip

The next lemma justifies the inductive step: suppose we have managed, by tame left action, to eliminate the terms of degree $2,\ldots, k-1$, then there is a sequence of elementary automorphisms such that their left action eliminates the term of degree $k$. This statement translates into the following lemma.

\begin{lem}
Let $\varphi$ be a polynomial automorphism such that
$$
\varphi(x_1)=x_1+f_1(x_1,\ldots,x_n)+r_1,\;\;\ldots,\;\;\varphi(x_n)=x_n+f_n(x_1,\ldots,x_n)+r_n
$$
and $f_i$ are homogeneous of degree $k$ and $r_i$ are the remaining terms (thus $\Ht(r_i)>k$). Then one can find a sequence $\sigma_1,\ldots, \sigma_m$ of tame automorphisms whose composition with $\varphi$ is given by
\begin{equation*}
\sigma_m\circ\ldots\circ\sigma_1\circ\varphi:\;x_1\mapsto x_1+F_i(x_1,\ldots,x_n)+R_1,\;\;\ldots,\;\;x_n\mapsto x_n+F_n(x_1,\ldots,x_n)+R_n
\end{equation*}
with $F_i$ homogeneous of degree $k+1$ and $\Ht(R_i)>k+1$.
\end{lem}
\begin{proof}
We will first show how to get rid of degree $k$ monomials in the images of all but one generator and then argue that the remaining image is rectified by an elementary automorphism.
Let $N\leq n$ be the number of images $\varphi(x_i)$ such that $f_i\neq 0$, and let $x_1$ and $x_2$ be two generators\footnote{Evidently, no loss of generality results from such explicit labelling.} corresponding to non-zero term of degree $k$. The image of $x_1$ admits the following presentation as an element of the polynomial ring $\mathbb{K}[x_3,\ldots,x_n][x_1,x_2]$:
\begin{equation*}
\varphi(x_1)= x_1+\sum_{d}\sum_{p+q=d}\lambda_{p,q}x_1^p x_2^q + r_i
\end{equation*}
where the coefficients $\lambda_{p,q}$ are polynomials of the remaining variables (thus the double sum above is just a way to express $f_1$ as a polynomial in $x_1$ and $x_2$ with coefficients given by polynomials in the rest of the variables).

Consider the transformation $\Phi_{\lambda\mu}$ of the following form
\begin{equation*}
x_1\mapsto x_1-\lambda (x_1+\mu x_2)^d,\;\;x_2\mapsto x_2-\lambda \mu^{-1}(x_1+\mu x_2)^d,\;\; x_3\mapsto x_3,\;\;\ldots,\;\; x_n\mapsto x_n,
\end{equation*}
with $\lambda \in \mathbb{K}[x_3,\ldots,x_n]$   and $\mu \in\mathbb{K}$.
This mapping is equal to the composition $\psi_{\mu}\circ\phi_{\lambda\mu}\circ\psi_{\mu}^{-1}$ with
\begin{equation*}
\psi_{\mu}:x_1\mapsto x_1+\mu x_2,\;\;x_2\mapsto x_2
\end{equation*}
and
\begin{equation*}
\phi_{\lambda\mu}: x_1\mapsto x_1,\;\;x_2\mapsto x_2+\lambda\mu^{-1}x_1^d
\end{equation*}
and so is a tame automorphism. As the ground field $\mathbb{K}$ has characteristic zero, it is infinite, so that we can find numbers $\mu_1,\;\ldots,\; \mu_{l(d)}$ such that the polynomials $$(x+\mu_1 y)^d,\;\ldots,\;(x+\mu_{l(d)} y)^d$$ form a basis of the $\mathbb{K}$-module of homogeneous polynomials in $x$ and $y$ of degree $d$ (this is an easy exercise in linear algebra). Therefore, by selecting $\Phi_{\lambda\mu}$ with appropriate polynomials $\lambda_{p,q}$ and $\mu_i$ corresponding to the basis, we eliminate, by acting with $\Phi_{\lambda\mu}$ on the left, the degree $d$ terms in the double sum. Iterating for all $d$, we dispose of $f_1$ entirely.

The above procedure yields a new automorphism $\tilde{\varphi}$ which is a composition of the initial automorphism $\varphi$ with a tame automorphism. The number $\tilde{N}$ of images of $x_i$ under $\tilde{\varphi}$ with non-zero term of degree $k$ equals $N-1$; therefore, the procedure can be repeated a finite number of times to give an automorphism $\varphi_1$, such that the image under $\varphi_1$ of only one generator contains a non-zero term of degree $k$. Let
\begin{equation*}
\varphi_1(x_n)=x_n+g_n(x_1,\ldots,x_n)+\tilde{r}_n
\end{equation*}
be the image of that generator (again, no loss of generality results from us having labelled it $x_n$). We claim now that the polynomial $g_n$ does not depend on $x_n$.

Indeed, otherwise the Jacobian of $\varphi_1$ (which must be a constant and is in fact equal to $1$ in our setting) would have a degree $k-1$ component given by
$$
\partial_{x_n} g_n(x_1,\ldots,x_n)\neq 0
$$
(remember that by construction $g_1=\ldots=g_{n-1}=0$), which yields a contradiction. Note that another way of looking at this condition is that if a polynomial mapping of the form
\begin{equation*}
x_1\mapsto x_1 + H_1(x_1,\ldots,x_n),\;\;\ldots,\;\;x_n\mapsto x_n + H_n(x_1,\ldots,x_n),\;\;\Ht(H_i)>1
\end{equation*}
is an automorphism, the higher-degree part $(H_1,\ldots, H_n)$ must have traceless Jacobian:
$$
\Tr \left(\frac{\partial H_i}{\partial x_j}\right)=0.
$$

Finally, since $g_n$ does not contain $x_n$, an elementary automorphism
\begin{equation*}
x_1\mapsto x_1,\;\;\ldots,\;\;x_{n-1}\mapsto x_{n-1},\;\;x_n\mapsto x_n-g_n(x_1,\ldots,x_n)
\end{equation*}
eliminates this term. The lemma is proved.
\end{proof}

The last lemma concludes the proof of Theorem 3.1 by induction. The proof of the inductive step is essentially a statement that a certain vector space invariant under a linear group action is, in a manner of speaking, big enough to allow for elimination by elements of the group. More precisely, let $T_{n,k}(\mathbb{K})$ be the vector space of all \textbf{traceless} $n$ by $n$ matrices whose entries are homogeneous of degree $k$ polynomials from $\mathbb{K}[x_1,\ldots,x_n]$, and let the group $\GL(n,\mathbb{K})$ act on $T_{n,k}$ as follows: for $A\in\GL(n,\mathbb{K})$ and $v\in T_{n,k}$, the image $A(v)$ is obtained by taking the product matrix $v A^{-1}$ and then performing (entry-wise in $vA^{-1}$) the linear change of variables induced by $A$. Then one has the following

\begin{prop}
If $V\subset T_{n,k}(\mathbb{K})$ is a $\mathbb{K}$-submodule invariant under the defined above action of $\GL(n,\mathbb{K})$, then either $V=0$ or $V= T_{n,k}(\mathbb{K})$.
\end{prop}

Properties of similar nature played an important role in \cite{K-BMR1, K-BMR2}. The invariance under linear group action will become somewhat more pronounced in the symplectomorphism case.

\section{Approximation by tame symplectomorphisms and lifting to Weyl algebra}
We turn to the proof of the more relevant to our context Theorem 3.2. The strategy is analogous to the proof of approximation for polynomial automorphisms with unit Jacobian, with a few more elaborate details which we now consider.

\medskip

The first step of the proof copies the polynomial automorphism case and takes the following form.

\begin{lem}
There is a linear transformation $A\in\Sp(2n,\mathbb{K})$
\begin{equation*}
(x_1,\;\ldots,\;x_n,\;p_1,\;\ldots,\;p_n)\mapsto (x_1,\;\ldots,\;x_n,\;p_1,\;\ldots,\;p_n)A
\end{equation*}
such that its composition $\sigma_A$ with $\sigma$ fulfills
\begin{equation*}
\Ht(\sigma_A(x_i)-x_i)\geq 2,\;\;\Ht(\sigma_A(p_i)-p_i)\geq 2
\end{equation*}
for all $i\in\lbrace 1,\ldots,n\rbrace$.
\end{lem}

\medskip

We now proceed to formulate the inductive step in the proof as the following main lemma.

\begin{lem}
Let $\sigma$ be a polynomial symplectomorphism such that
\begin{equation*}
\sigma(x_i)=x_i+U_i,\;\;\sigma(p_i)=p_i+V_i
\end{equation*}
and $U_i$ and $V_i$ are of height at least $k$. Then there exists a tame symplectomorphism $\sigma_k$ such that the polynomials $\tilde{U}_i=(\sigma_k^{-1}\circ\sigma)(x_i)-x_i$ and $\tilde{V}_i=(\sigma_k^{-1}\circ\sigma)(p_i)-p_i$ are of height at least $k+1$.
\end{lem}
\begin{proof}
In order to establish the inductive step, we are going to need the following lemma.

\begin{lem}
Suppose $\mathbb{K}$ is an infinite field, $A=\mathbb{K}[x_1,\ldots,x_n,p_1,\ldots,p_n]$ is the polynomial algebra with standard $\mathbb{Z}$-grading according to the total degree
\begin{equation*}
A=\bigoplus_{d\geq 0} A_d, \;\; A_d=\lbrace \text{homogeneous polynomials of total degree d}\rbrace.
\end{equation*}
Let $V$ be a $\mathbb{K}$-submodule of $A$ invariant under the action of $\Sp(2n, \mathbb{K})$ (given by linear symplectic changes of variables). Suppose $V$ is contained in a given homogeneous component $A_d$. If $V \neq 0$ then $V = A_d$.
\end{lem}

\begin{proof}
We first observe that if $f=\sum_{l}f_l$ is a non-zero polynomial given by the sum of degree $d$ monomials in $x_i$, $p_j$, then $f\in V$ implies $f_l\in V$ for all $l$. Indeed, consider a linear  symplectomorphism $\Lambda$ of the form
\begin{equation*}
x_i\mapsto \lambda_i x_i,\;\;p_i\mapsto \lambda_i^{-1}p_i,\;\; \lambda_i\neq 0.
\end{equation*}
Then $f\mapsto \sum_{l}\prod_{i=1}^{n}\lambda_i^{k_{li}}f_l\in V$. As the ground field $\mathbb{K}$ is infinite and $f_l$ are linearly independent, we can take sufficiently many automorphisms of the form $\Lambda$ in order to produce a basis $\lbrace \Lambda_1 f,\ldots, \Lambda_N f\rbrace$ of the span of $f_l$. Since by our assumptions $\Lambda f\in V$, the observation follows.

\smallskip

Next we observe that if $V\neq 0$ and $V\in A_d$, then every monomial of the form $x_i^d$, $p_j^d$ belongs to $V$. For if $f\in V$ is a non-zero polynomial, then $f$ has in its decomposition a monomial of the form $a x_1^{k_1}\ldots x_n^{k_n}p_1^{m_1}\ldots p_n^{m_n}$ with non-zero $a\in \mathbb{K}$. To prove that, say, $p_1\in V$, one needs to apply a sequence of linear symplectomorphisms to $f$ so that the image of $x_1^{k_1}\ldots x_n^{k_n}p_1^{m_1}\ldots p_n^{m_n}$ is a sum that contains $p_1^d$. This is accomplished by means of the following procedure. First, we get rid of every $x_i$ by taking in succession
\begin{equation*}
x_i\mapsto x_i+p_i, \;\; p_i\mapsto p_i \;\;\text{(other generators fixed)}
\end{equation*}
and using the above homogeneity statement to single out the monomial with maximal degree in $p_i$. Thus we obtain that $p_1^{d_1}\ldots p_n^{d_n}\in V$ for some $d_1,\ldots, d_n$, $d_1+\cdots+d_n=d$. We then dispose of $p_2,\ldots, p_n$ by applying symplectomorphisms of the form (written down for $p_2$)
\begin{equation*}
x_1\mapsto x_1-x_2,\;x_2\mapsto x_1+x_2,\; p_1\mapsto p_1+p_2,\; p_2\mapsto p_2-p_1
\end{equation*}
and again using homogeneity to single out the monomial with $p_1$. A procedure identical to the above is applied to show that $x_i^d\in V$, $i=1,\ldots, n$.

\smallskip

What remains to prove is that the mixed terms $x_1^{k_1}\ldots x_n^{k_n}p_1^{m_1}\ldots p_n^{m_n}$ (where at least two powers are non-zero) are in $V$. This can be done by looking at linear combinations of $x_i^d$ and $p_j^d$ and applying suitable symplectomorphisms in a manner similar to our previous construction and using the homogeneity argument. We leave details to the reader.
\end{proof}

\medskip

We now turn to the proof of the inductive step.  Suppose that
\begin{equation*}
\sigma: x_i\mapsto x_i + f_i + P_i,\;\;p_j\mapsto p_j + g_j + Q_j
\end{equation*}
is a polynomial symplectomorphism, where $f_i$ and $g_j$ are degree $k$ components and the height of $P_i$ and $Q_j$ is greater than $k$. The preservation of the symplectic structure by $\sigma$ means that the $k$-th component obeys the following identities:
\begin{equation*}
\lbrace x_i, f_j\rbrace - \lbrace x_j, f_i\rbrace=0
\end{equation*}
and
\begin{equation*}
\lbrace p_i, f_j\rbrace - \lbrace p_j, f_i\rbrace=0
\end{equation*}
where $\lbrace\;,\;\rbrace$ is the Poisson bracket corresponding to the symplectic form. In the case of standard symplectic structure these identities translate into
\begin{equation*}
\frac{\partial f_i}{\partial p_j}-\frac{\partial f_j}{\partial p_i}=0,\;\; \frac{\partial g_i}{\partial x_j}-\frac{\partial g_j}{\partial x_i}=0,
\end{equation*}
in which one recognizes the condition for an appropriate differential form to be closed. The triviality of affine space cohomology then implies that there exists a polynomial $F(x_1,\ldots,x_n,p_1,\ldots,p_n)$, homogeneous of degree $k+1$, such that
\begin{equation*}
\frac{\partial F}{\partial p_i}=f_i,\;\; \frac{\partial F}{\partial x_i}=g_i;
\end{equation*}
in this way the $k$-component of a symplectomorphism is generated by a homogeneous polynomial.
The tame symplectomorphism group acts on the space of all such generating polynomials (the image of a polynomial is the polynomial corresponding to the $k$-component of the composition with the tame symplectomorphism), and the orbit of this tame action carries the structure of a $\mathbb{K}$-module (one may easily come up with a symplectomorphism corresponding to the sum of two generating polynomials). Therefore this space fulfills the conditions of the previous lemma, which in this case implies that one can, by a composition with a tame symplectomorphism, eliminate the $k$-component. The main lemma, and therefore the Theorem 3.2, is proved.
\end{proof}

\medskip

Once the approximation for the case of symplectomorphisms has been established, we can investigate the problem of lifting symplectomorphisms to Weyl algebra automorphisms. More precisely, one has the following
\begin{prop}
Let $\mathbb{K}=\mathbb{C}$ and let $\sigma:P_n(\mathbb{C})\rightarrow P_n(\mathbb{C})$ be a symplectomorphism over complex numbers. Then there exists a sequence
\begin{equation*}
\psi_1,\;\psi_2,\;\ldots,\;\psi_k,\;\ldots
\end{equation*}
of tame automorphisms of the $n$-th Weyl algebra $W_n(\mathbb{C})$, such that their images $\sigma_k$ in $\Aut P_n(\mathbb{C})$ approximate $\sigma$.
\end{prop}
\begin{proof}
This is an immediate corollary of Theorem 3.2 and the existence of tame subgroup isomorphism \cite{BKK1}.

\end{proof}
A few comments are in order. First, the quantization of elementary symplectomorphisms is a very simple procedure: one needs only replace the $x_i$ and $p_i$ by their counterparts $\hat{x}_i$ and $\hat{p}_i$ in the Weyl algebra $W_n$. Because the transvection polynomials $f$ and $g$ (in the expressions for elementary symplectomorphisms) depend, as it has been noted, on one type of generators (resp. $p$ and $x$), the quantization is well defined.

Second, as the tame automorphism groups $\TAut W_n(\mathbb{C})$ and $\TAut P_n(\mathbb{C})$ are isomorphic, the correspondence between sequence of tame symplectomorphisms converging to symplectomorphisms and sequences of tame Weyl algebra automorphisms is one to one. The main question is how one may interpret these sequences as endomorphisms of $W_n(\mathbb{C})$.

\smallskip

Our construction shows that these sequences of tame automorphisms may be thought of as (vectors of) power series -- that is, elements of
\begin{equation*}
\mathbb{C}[[\hat{x}_1,\ldots,\hat{x}_n,\hat{p}_1,\ldots,\hat{p}_n]]^{2n}.
\end{equation*}

The main problem therefore consists in verifying that these vectors have entries polynomial in generators -- that is, that the limits of lifted tame sequences are Weyl algebra endomorphisms.

\medskip

One could take a more straightforward (albeit an equivalent) approach to the lifting of symplectomorphisms by following the prescription of deformation quantization: starting with a symplectic automorphism of the polynomial algebra $A=\mathbb{K}[x_1,\ldots,x_n,p_1,\ldots,p_n]$, one constructs a map of $A[[\hbar]]$, the algebra of formal power series (in Planck's constant $\hbar$), which preserves the star product satisfying Weyl algebra identities. The approximation theory as developed in this text is then a property of the $\hbar$-adic topology. The (algebraically closed version of) Conjecture 1.3 would then follow if one were to establish a cutoff theorem.

\smallskip

In our closing remark we comment on the remaining two main theorems as stated in the introduction. As the reader may infer from the proof of approximation theorems, the approximation in formal power series topology is natural in the sense that agrees with the $\mathfrak{m}$-adic topology in the local ring generated by the coefficients of the approximated automorphism. More precisely, we have the following property, formulated separately for the two cases we consider in the paper.

\begin{thm}
Let $\varphi$ be a polynomial automorphism and let $\mathcal{O}_{\varphi}$ be the local ring of $\Aut \mathbb{C}[x_1,\ldots,x_n]$ with its maximal ideal $\mathfrak{m}$. Then there exists a tame sequence $\lbrace \psi_k\rbrace$ which converges to $\varphi$ in power series topology, such that the coordinates of $\psi_k$ converge to coordinates of $\varphi$ in $\mathfrak{m}$-adic topology.
\end{thm}

\begin{thm}
Let $\sigma$ be a symplectomorphism and let $\mathcal{O}_{\sigma}$ be the local ring of $\Aut P_n(\mathbb{C})$ with its maximal ideal $\mathfrak{m}$. Then there exists a sequence of tame symplectomorphisms $\lbrace \sigma_k\rbrace$ which converges to $\sigma$ in power series topology, such that the coordinates of $\sigma_k$ converge to coordinates of $\sigma$ in $\mathfrak{m}$-adic topology.
\end{thm}  

These are our last two main results.

\section{Conclusion}

We have developed tame approximation theory for symplectomorphisms in formal power series topology. By virtue of the known correspondence between tame automorphisms of the even-dimensional affine space and tame automorphisms of the Weyl algebra, which is the object corresponding to the affine space in terms of deformation quantization, we have arrived at the lifting property of symplectomorphisms. This line of research may yield new insights into endomorphisms of the Weyl algebra, the Dixmier conjecture, and the Jacobian conjecture.

\section*{Acknowledgments}
This paper is supported by the Russian Science Foundation grant No. 17-11-01377.

\end{document}